\documentclass[preprint,11pt]{elsarticle}
\usepackage{mathtools}

\usepackage{amsmath}
\usepackage{nicefrac}
\usepackage{booktabs}
\usepackage{multirow}
\usepackage{amssymb}
\usepackage{amsfonts,latexsym}
\usepackage{epsfig}
\usepackage{amsthm}
\usepackage{epstopdf}
\usepackage{graphicx}
\usepackage{color}
\usepackage{comment}
\usepackage{ulem}
\usepackage{tikz}
\usepackage{float}
\usepackage{hyperref}
\usepackage{algorithm}
\usepackage{algorithmic}

\newcommand{\calK}{\mathcal{K}}

\usepackage{mathtools}
\usepackage[scr=boondoxupr]{mathalpha}

\usetikzlibrary{decorations.markings}
\usepackage[top=3cm,bottom=4cm,right=3cm,left=3cm]{geometry}
\numberwithin{equation}{section}
\newtheorem{exe}{Example}
\newtheorem{theorem}{Theorem}[section]
\newtheorem{proposition}{Proposition}[section]

\newtheorem{definition}{Definition}[section]
\newtheorem{remark}{Remark}[section]

\journal{arXiv preprint}
\begin{document}
\begin{frontmatter}
\title{A randomized global GMRES method for linear systems with multiple right-hand sides}

\author[1]{A. Badahmane}
\ead{badahmane.achraf@gmail.com}
\author[1]{I. Ezzagouri}
\ead{Issa.EZZAGOURI@um6p.ma}
\author[2]{Xian-Ming Gu\corref{cor1}}
\ead{guxianming@live.cn}
\cortext[cor1]{Corresponding author.}
\address[1]{The UM6P Vanguard Center, Mohammed VI Polytechnic  University  (UM6P), Rocade Rabat-Sal\'{e}, Technopolis, 11103, Morocco}
\address[2]{School of Mathematics, Southwestern University of Finance and Economics, Chengdu,
Sichuan 611130, P.R. China}
%%%%%%%%%
\begin{abstract}
In this paper, we develop a new Randomized Global Generalized Minimum Residual (RGl-GMRES) algorithm for efficiently computing solutions to large-scale linear systems with multiple right-hand sides. The proposed method builds on a recently developed randomized global Gram–Schmidt process, in which sketched Frobenius inner products are employed to approximate the exact Frobenius inner products of high-dimensional matrices. We give some new convergence results of the RGl-GMRES method for such multiple linear systems. In the case where the coefficient matrix $A$ is diagonalizable, we derive new upper bounds for the randomized Frobenius norm of the residual. In this paper, we study how to introduce matrix sketching in this algorithm. It allows us to reduce the dimension of the problem in one of the main steps of the algorithm. We also show how to apply the RGl-GMRES method for solving the Lyapunov matrix equation. To validate the effectiveness and practicality of this approach, we conduct several numerical experiments, which demonstrate that our RGl-GMRES method is competitive with the global GMRES method for solving large-scale multiple linear systems and Lyapunov matrix equation.
\end{abstract}

\begin{keyword}
  % 7. Keywords that describe the paper
  % Type a list of keywords (also known as key phrases or key terms), one per line to characterize your submission. You should specify at least three keywords.
Multiple right-hand sides; Global GMRES; Randomized sketching; Global Gram-Schmidt process; Lyapunov matrix equation
\end{keyword}
\end{frontmatter}

% REQUIRED
%\begin{MSCcodes}
%65F22 % Ill-posedness, regularization
%65F10 % Iterative methods for linear systems
%65K10 % optimization and variational techniques
%15A29 %Inverse problems in linear algebra
%\end{MSCcodes}
\section{Introduction} % (we should extend this section and include the recent developments of the randomized Krylov subspace methods)
In many scientific and engineering simulations (see, e.g., \cite{Puzyrev2015,Lee2017,JELICH2021}), we need to solve linear systems with multiple right-hand sides:
\begin{equation}
\label{MRHS}
AX=B,
\end{equation}
where \(A \in \mathbb{R}^{n \times n}\) is a large, sparse, nonsymmetric real matrix, and \(B, X \in \mathbb{R}^{n \times s}\) with \(s \ll n\). Sparse direct methods are highly efficient for medium-sized problems via matrix factorization (e.g., LU decomposition); however, they become computationally prohibitive as \(n\) grows large \cite{PUZYREV2016}. Alternatively, applying standard modern iterative methods (e.g., Krylov subspace methods) to each right-hand side individually is inefficient because computational information is not shared across systems \cite{Gutknecht2007}. To overcome this limitation, block versions of iterative methods are preferred, as they generate iterates for all systems simultaneously and leverage shared subspace information. They are attractive not only because of this numerical feature (larger search subspace), but also from a computational view
point as they enable the use of BLAS3 like implementation.

Over the past few decades, various block Krylov subspace methods have been developed for multiple linear systems \eqref{MRHS}. Early work includes the block biconjugate gradient (Bl-BCG) algorithm \cite{OLEARY1980}, its stabilized variant \cite{Simoncini97}, and the block BiCGSTAB method \cite{Guennouni03}. Based on the block Arnoldi process, the block generalized minimal residual (Bl-GMRES) algorithm was established in \cite{Vital1990} and further investigated in \cite{SIMONCINI96,SIMONCINI1996,MORGAN2005,Baker2006,Calandra2013}. Using the block Bi-Lanczos process, Freund and Malhotra \cite{FREUND1997} developed the block quasi-minimal residual (BQMR) algorithm. Furthermore, Guennouni et al. \cite{ELGUENNOUNI2004} derived block Lanczos-type methods (such as block Lanczos/Orthodir and BIODIR) using matrix-valued orthogonal polynomials and Schur complements, while further advances in block minimal residual methods were reported in \cite{BOUYOULI2007}. To handle linearly or near-linearly dependent vectors during iterations, deflation procedures are widely used \cite{FREUND1997,Gutknecht2007}. To address inexact breakdowns and improve efficiency, Robbé and Sadkane \cite{ROBBE2006} retained directions from near-converged solutions, whereas Agullo et al. \cite{Agullo2014} combined inexact breakdown detection with deflated restarting specifically for block GMRES.

% The seed methods are introduced by some authors to solve the matrix equation \eqref{MRHS}. This class of methods consists of selecting a single system (called the seed system), generating the Krylov subspace by the Arnoldi or Bi-Lanczos process, and projecting the residuals of other systems onto the Krylov subspace. The procedure is repeated with another seed system until all the systems are solved. The methods have been developed in [21, 19]. The seed seed GMRES method is studied in [19], [13], respectively. Lotstedt and Nilsson [22] proposed a minimal residual interpolation method for linear systems with multiple right-hand sides.

To solve multiple linear systems \eqref{MRHS}, seed methods offer an efficient alternative by selecting a primary system to build a shared Krylov subspace and projecting other residuals onto it sequentially. While initially investigated for symmetric positive definite systems \cite{Smith1989,Chan1997}, this strategy was subsequently adapted to non-Hermitian cases through algorithms such as Seed-GMRES \cite{Simoncini1995} and Seed-BiCGSTAB \cite{Jbilou2016}. Over the years, sophisticated extensions have been proposed to accelerate convergence, notably the seed weighted GMRES method \cite{Elbouyahyaoui18} and seed GMRES augmented with eigenvectors \cite{Gu2002}.

Recently, the global Krylov subspace methods have been generated by projecting globally the initial matrix residual onto a matrix Krylov subspace. Using the global Arnoldi process, Jbilou et al. \cite{JBILOU1999} presented the global full orthogonalization method (Gl-FOM) and the global GMRES method (Gl-GMRES) for multiple linear systems \eqref{MRHS}. Since then, some global methods based on global Arnoldi or Hessenberg process have been used to solve multiple linear systems \eqref{MRHS} \cite{Heyouni2001,Heyouni2005,LIN2005,Du2008}. In the case where the coefficient matrix 
 is diagonalizable, Bellalij et al. \cite{BELLALIJ2008} obtained some new convergence results of Gl-GMRES for nonsymmetric matrix equations.  Using global Bi-Lanczos process, Jbilou et al. \cite{Jbilou2005} proposed the global biconjugate gradient method (Gl-BCG) and developed one of the most effective variants of Gl-BCG, named as the global BiCGSTAB algorithm. In order to avoid multiplication by the transpose of a matrix, Zhang and Dai \cite{Zhang2008} introduced the global conjugate gradient squared method, which, along with the global TFQMR method presented in \cite{Zhang2013}, demonstrated superior performance to the Gl-BCG method.

In recent years, randomized Krylov subspace methods have become increasingly popular in the field of iterative algorithms for solving large-scale linear systems and (generalized) eigenvalue problems \cite{Martinsson2020,Nakatsukasa2024,Saibaba2025}. These methods employ a randomized QR-based Arnoldi process to generate a non-orthogonal basis of the Krylov subspace, thereby reducing computational costs and enabling the development of various Krylov subspace algorithms, such as the randomized GMRES method \cite{Balabanov22}. Furthermore, for instance, Balabanov and Grigori \cite{Balabanov2025} presented a randomized block Gram-Schmidt process to construct the randomized Bl-GMRES method for multiple linear systems \eqref{MRHS}. While it retains the theoretical convergence bound of classical Bl-GMRES (converging in at most $\lceil n/s \rceil$ iterations), its computational cost escalates rapidly as the number of right-hand sides $s$ increases. In contrast, the global GMRES (Gl-GMRES) method with further investigations \cite{BELLALIJ2008} utilizes a global Arnoldi process, making its computational complexity significantly less sensitive to $s$ \cite{JBILOU1999}.

Compared to alternative solver classes for multiple linear systems, global Krylov subspace methods offer distinct advantages in numerical robustness. Specifically, the above block Krylov subspace methods are prone to (near) linear dependence among the columns of residual matrices, which often causes slow convergence or algorithm breakdown \cite{Langou2003}. Meanwhile, seed Krylov subspace methods are primarily effective when the right-hand sides vary slowly across systems \cite{Elbouyahyaoui18}. Global Krylov subspace methods circumvent both limitations, making them applicable to a much broader range of problems. Nevertheless, when the number of right-hand sides $s$ is extremely large, standard Gl-GMRES faces its own severe bottleneck: the heavy computational and memory burden of Frobenius inner products and orthogonalization procedures during global Krylov subspace generation. This challenge provides a compelling motivation to incorporate randomization techniques into the global Krylov subspace framework, aiming to significantly alleviate computational overhead while preserving the superior stability and convergence behavior of Gl-GMRES. Furthermore, as a natural extension, we demonstrate that the proposed randomized global framework can be readily applied to solve continuous-time Lyapunov matrix equations of the form $AX + XA^T = -BB^T$.

To this end, in this work, we propose the Randomized Gl-GMRES (RGl-GMRES) method. By substituting costly exact Frobenius inner products with randomized estimators at key stages of the proposed algorithm, RGl-GMRES achieves substantial computational savings without compromising theoretical convergence properties. The primary objective of this paper is to establish the RGl-GMRES framework for large-scale multiple linear systems \eqref{MRHS} and to thoroughly analyze its convergence characteristics both theoretically and numerically.

% One might consider applying methods for a single vector, such as those described in , to each problem  $A x_{i} = b_{i}$, it is
% well known for linear systems.
% Instead of applying a standard iterative method to the solution of each one of the  problems with several right-hand sides:
% \begin{eqnarray}
% \label{MRHS}
%     A x_{i} = b_{i}, \quad i = 1, \ldots, s,
% \end{eqnarray}
% independently, it is often more efficient to apply a global method to \eqref{MRHS}. However, that global Krylov approaches treating all columns $b_{i}$ at once can be computationally advantageouss; see, e.g. \cite{Jbilou2,BELLALIJ}. 

The rest of this paper is organized as follows. Section \ref{sec2} is devoted to presenting the
preliminary concepts and definitions of the Gl-GMRES method. The proposed RGl-GMRES algorithm is outlined in Section \ref{sec:randomized}. In Section \ref{sec4x} we apply the RGl-GMRES for solving the Lyapunov matrix equation. Section \ref{sec4} reports numerical simulations for both the classical Gl-GMRES method and its randomized variant. Finally, Section \ref{sec5} concludes the paper with recommendations on RGl-GMRES usage tailored to specific problem traits.

\textbf{Notations.}
Let \( \mathbb{R}^{n\times s} \) denote the space of real \(n\times s\) matrices. 
For \( Y, Z \in \mathbb{R}^{n\times s} \), we define the \textbf{Frobenius inner product} as
\[
\langle Y, Z \rangle_F = \operatorname{Trace}\bigl(Y^T Z\bigr).\] 
The associated norm is the \textbf{Frobenius norm}, denoted by $\|X\|_{F}=\sqrt{\mathrm{Trace}\left(X^{T}X\right)}$ for $X\in\mathbb{R}^{n\times s}$ .
A set of vectors (or matrices) in \(\mathbb{R}^{n\times s}\) is said to be \textbf{\(F\)-orthonormal} 
if it is orthonormal with respect to the inner product \(\langle \cdot, \cdot \rangle_F\). Finally, the \textbf{Kronecker product} of matrices \(C\) and \(D\) is defined as
\[
C \otimes D = [\,c_{i,j} D\,],
\]
where \(c_{i,j}\) are the entries of \(C\). 
\section{Global GMRES (Gl-GMRES) method}
\label{sec2}
As described above, it is well-known that the global Krylov subspace methods (e.g. Gl-GMRES) outperform other block iterative methods for solving large-scale multiple linear systems \eqref{MRHS} when the coefficient matrix is large and nonsymmetric. In this section, we briefly review the concept of the global GMRES method~\cite{Badahmane2019,JBILOU1999} for solving such linear systems.
%\begin{definition}~(\cite{Badahmane2019,Jbilou2005}).
%Let $\mathcal{X}$ and $\mathcal{Y}$ are two matrices of dimension $n\times s$, we consider the scalar product 
%\begin{eqnarray}
%<\mathcal{X},\mathcal{Y}>_{F}=\mathrm{Trace}(\mathcal{X}^{T}\mathcal{Y}),
%\end{eqnarray}
%where $\text{Trace}$ is the trace of the square matrix $\mathcal{X}^{T}\mathcal{Y}$, refers to the sum of its diagonal entries.
%The associated norm is the Frobenius norm, denoted $\|.\|_{F}$, which can be expressed as follows:
%\begin{eqnarray}
% \label{normF}
%
%\end{eqnarray}
%\end{definition}
\begin{definition}
[Global Krylov subspace \cite{JBILOU1999,Badahmane2019}] Let $V\in\mathbb{R}^{n\times s}$, the global Krylov subspace $\mathcal{K}_{k}(A,V)$  is the
	subspace  spanned by the  matrices $\{V, AV,\cdots,A^{k-1}V\}$.
\end{definition}
\begin{remark}
\label{Kry}
Let $V\in\mathbb{R}^{n\times s}$. By the definition of the global Krylov subspace $\mathcal{K}_{k}(A,V)$, any matrix $\mathcal{Z}\in\mathcal{K}_{k}(A,V)$ can be expressed as
\begin{eqnarray}
\mathcal{Z}\in \mathcal{K}_{k}(A,V)\iff \mathcal{Z}=\sum_{i=1}^{k}\psi_{i}A^{i-1}V,~~
\psi_{i}\in\mathbb{R},~i=1,\ldots,k.  \nonumber
\end{eqnarray}
In other words,  $\mathcal{K}_k(A,V)$ consists of all matrices in $\mathbb{R}^{n\times s}$ that can be represented in the form $\mathcal{Z}=\Theta_k(A)V$, where $\Theta_k(\cdot)$ is a polynomial of degree at most equal $k-1$. \cite{JBILOU1999,Badahmane2019}
\end{remark}
% \begin{definition}
% 	The Kronecker product of $M$ and $J$, where $M$ is a matrix with dimension $m \times r$ and $J$ is a matrix with dimensions $n \times s$, can be expressed in a specific mathematical form
% 	\begin{eqnarray}
%  \label{Kronecker}
% 	M\otimes J=\left(\begin{array}{ccc}
% 	M_{11 }J &\hdots  &  M_{1r}J\\
% 	\vdots &\ddots &\vdots    \\
% 	M_{m1}J & \hdots &  M_{mr}J
% 	\end{array}\right)\in\mathbb{R}^{m n\times rs}.
%    \end{eqnarray}
% \end{definition}
\begin{definition}[Diamond Product]\label{Diamond}
	Let $Y$ and $\mathcal{Z}$ be  matrices of  dimensions $n_{u}\times l s$ and $n_{u}\times l r$, respectively. The matrices $Y$ and $\mathcal{Z}$ are constructed by concatenating $n_{u}\times l$ matrices $Y_{i}$ and $\mathcal{Z}_{j}$ (for $i = 1,...., s$ and $j = 1,...,r)$. The symbol 
 $\diamond$ represents a specific product defined by the following formula:
	$$ Y^T
	\diamond \mathcal{Z}=
	\left(
	\begin{array}{cccc}
	\langle Y_{1},\mathcal{Z}_{1} \rangle_{F} & \langle Y_{1},\mathcal{Z}_{2} \rangle_{F} & \ldots & \langle Y_{1},\mathcal{Z}_{r} \rangle_{F} \\
	\langle Y_{2},\mathcal{Z}_{1} \rangle_{F}&\langle Y_{2},\mathcal{Z}_{2} \rangle_{F}&\ldots&\langle Y_{2},\mathcal{Z}_{r} \rangle_{F}\\
	\vdots&\vdots&\ddots&\vdots\\
	\langle Y_{s},\mathcal{Z}_{1} \rangle_{F}&\langle Y_{s},\mathcal{Z}_{2} \rangle_{F}&\ldots&\langle Y_{s},\mathcal{Z}_{r} \rangle_{F}
	\end{array}
	\right)\in\mathbb{R}^{s\times r}.$$ 
    This product is commonly used in the formulation of global Krylov
subspace methods; see, e.g., \cite{JBILOU1999,Badahmane2019}.
\end{definition}
%%%%%%%%%%%%%%
\subsection{Gl-Arnoldi process and Gl-GMRES method}
\label{Gl-Arnoli}
We use the modified global Arnoldi (Gl-Arnoldi) process \cite[Algorithm 2.2]{JBILOU1999}, which is based on the modified
Gram--Schmidt process, for building an $F$-orthonormal basis $V_1, V_2, \ldots, V_k$ of $\mathcal{K}_k(A,V)$. In other words,
\begin{equation}
\operatorname{Trace}(V_i^T V_j) = \delta_{ij}, 
\quad i,j = 1,\cdots,k,
\end{equation}
where $\delta_{ij}$ is the Kronecker symbol. 

Let $X_0$ be the initial approximate solution of Eq. \eqref{MRHS} and let $R_0 = B - AX_0$ be the corresponding residual. Before describing the global GMRES methods, we give
some property of the previous algorithm.
%\begin{algorithm}
%\caption{Gl-Arnoldi process}
%\begin{algorithmic}[1]
%\STATE $V_1 = R_0/\|R_0\|_F$
%\FOR{$j = 1,2,\ldots,k$}
 % \STATE $W := AV_j$
  %\FOR{$i = 1,2,\ldots,j$}
   % \STATE $h_{i,j} = \langle W, V_i \rangle_F$
    %\STATE $W = W - h_{i,j} V_i$
  %\ENDFOR
  %\STATE $h_{j+1,j} = \|W\|_F$
  %\STATE $V_{j+1} =W/h_{j+1,j}$
%\ENDFOR
%\end{algorithmic}
%\end{algorithm}

%\medskip
%\noindent
\begin{proposition}
Assume that $h_{i+1,i} \neq 0$ for $i = 1,\ldots,k$, and let
\[
\mathcal{V}_k = [V_1, V_2, \ldots, V_k],
\]
where $V_i$, for $i=1,\ldots,k$, are the matrices computed by Gl-Arnoldi process.
Then the following relations hold:
\begin{align}
\label{arno}
A\mathcal{V}_k &= \mathcal{V}_k\diamond H_k + h_{k+1,k}
      [\, O, \ldots, O, V_{k+1} \,],  \nonumber \\
A\mathcal{V}_k &= \mathcal{V}_{k+1}\diamond \bar{H}_k,
\end{align}
where $H_k \in \mathbb{R}^{k \times k}$ is an upper Hessenberg matrix whose entries are \(h_{i,j} = \operatorname{Trace}(V_i^T A V_j)\) and
$\widetilde{V}_k = [V_1, \ldots, V_{k+1}]$.
\[
\mathcal{V}_{k+1} = [V_1, \ldots, V_k, V_{k+1}] \in \mathbb{R}^{n\times (k+1)s},
\qquad
\bar{H}_k \in \mathbb{R}^{(k+1)\times k}.
\]
\end{proposition}

The global GMRES method constructs, at step $k$, the approximation $X_k$
%\begin{comment}
satisfying the following two relations:
\begin{equation}
X_k - X_0 \in \mathcal{K}_k(A,R_0),
\qquad
\langle A^j R_0, R_k \rangle_F = 0,
\quad j = 1,\ldots,k.
\end{equation}
The residual $R_k = B - AX_k$ satisfies the minimization property
\begin{equation}
\label{eq:gmres_min}
\|R_k\|_F
= \min_{Z \in \mathcal{K}_k(A,R_0)} \|R_0 - AZ\|_F,
%\tag{4.1}
\end{equation}
which is equivalent to solving the following least squares problem:
\begin{equation}
\|R_k\|_F = \min_{y \in \mathbb{R}^k} \left\| \beta e_1^{(k+1)} - \bar{H}_k y \right\|_2,\quad \|R_0\|_F = \beta,
\label{LS2}
\end{equation}
where \(\bar{H}_k\) is the \((k + 1) \times k\) upper Hessenberg matrix and \(e_1^{(k+1)}\) is the first canonical basis vector in \(\mathbb{R}^{k+1}\) and \(y\) is a vector of \(\mathbb{R}^k\). A detailed description of this method is given in \cite{JBILOU1999,LIN2005}. 

It is worth noting that for Gl-GMRES, we have to solve only one $k$-dimensional least-squares problem \eqref{LS2}
to get the $k$-dimensional vector $y_k$. This leads to considerable savings compared to the Bl-GMRES
algorithm where the Hessenberg matrix $H_k$ is of dimension $ks \times ks$ \cite{Vital1990} or to the simple GMRES algorithm
where $s$ least-squares problems have to be solved at each iteration \cite{JBILOU1999}. Moreover, we recall that the value of $k$ is limited by the storage constraints and by avoiding
rounding errors. Hence we have to use a restarting strategy to address such
problems. 
%\end{comment}
%\begin{comment}
\subsection{Convergence analysis of the Gl-GMRES method}
\label{sec2.2}
In this section, we recall the main convergence results for the Gl-GMRES method established in \cite{BELLALIJ2008}. 
Let $A = Z \Lambda Z^{-1}$ be the spectral decomposition of $A$, where $\Lambda = \operatorname{diag}(\lambda_1, \dots, \lambda_n)$ contains the eigenvalues of $A$ and $Z$ is the corresponding eigenvector matrix. 
Let the initial residual be decomposed as $R_0 = Z\beta$, where $\beta = [\beta^{(1)}, \dots, \beta^{(s)}] \in \mathbb{R}^{n \times s}$. 

Define the normalized weight vector $\gamma = (\gamma_1, \dots, \gamma_n)^T$ by
\begin{equation}
\label{eq:gamma_gl}
\gamma_i = \frac{\sum_{l=1}^s |\beta_i^{(l)}|^2}{\|\beta\|_F^2}, \quad i = 1, \ldots, n,
\end{equation}
and let $D_\gamma = \operatorname{diag}(\gamma_1, \ldots, \gamma_n)$ with $\gamma_i \ge 0$ and $\sum_{i=1}^n \gamma_i = 1$. The Krylov--Vandermonde matrix $\overline{V}_{k+1} \in \mathbb{C}^{n \times (k+1)}$ is defined as
\begin{equation}
\label{eq:vandermonde}
\overline{V}_{k+1} = 
\begin{bmatrix}
1 & \lambda_1 & \cdots & \lambda_1^k \\
\vdots & \vdots & & \vdots \\
1 & \lambda_n & \cdots & \lambda_n^k
\end{bmatrix}.
\end{equation}
%%%%%%%%%%%%%%%%%%
\begin{theorem}[\cite{BELLALIJ2008}]
\label{thm:gl_convergence}
Let $R_k = B - AX_k$ be the $k$-th residual obtained by the Gl-GMRES method applied to~\eqref{MRHS}. Let $\mathcal{P}_k$ denote the set of polynomials of degree at most $k$ satisfying $p(0) = 1$, and let $\kappa_2(Z) = \|Z\|_2 \|Z^{-1}\|_2$. Then
\begin{equation}
\label{eq:gl_bound_gamma}
\frac{\|R_k\|_F^2}{\|\beta\|_F^2} \le \frac{\|Z\|_2^2}{e_1^T (\overline{V}_{k+1}^H D_\gamma \overline{V}_{k+1})^{-1} e_1},
\end{equation}
and consequently,
\begin{equation}
\label{eq:gl_minmax_bound}
\frac{\|R_k\|_F}{\|R_0\|_F} \le \kappa_2(Z) \min_{p \in \mathcal{P}_k, p(0)=1} \max_{\lambda \in \operatorname{Sp}(A)} |p(\lambda)|,
\end{equation}
where $\operatorname{Sp}(A)$ denotes the spectrum of $A$.
\end{theorem}

\section{RGl-GMRES method}
\label{sec:randomized}
Although global Krylov subspace methods \cite{LIN2005,BELLALIJ2008,JBILOU1999} are effective for solving large-scale multiple linear systems~\eqref{MRHS}, including saddle point problems arising in constrained optimization, fluid dynamics, and mixed finite element discretizations \cite{Badahmane2019}, the computational overhead of orthonormalization and exact Frobenius inner products among matrix-valued basis blocks remains a major bottleneck as the problem dimension increases. To alleviate this cost, randomized sketching techniques have recently been introduced into Krylov subspace frameworks \cite{Tropp2022,Nakatsukasa2024}. In this section, we develop the RGl-GMRES method by embedding random sketching into the global Arnoldi process, analyze its theoretical convergence behavior, and establish its vectorized equivalence to classical randomized GMRES.

\subsection{Introduction to random sketching}
Let $\Theta \in \mathbb{R}^{\ell \times n}$ with $\ell \ll n$ be a sketching matrix. This matrix can be viewed as an embedding of subspaces of $\mathbb{R}^n$ into lower-dimensional subspaces of $\mathbb{R}^\ell$, and is referred to as an $\ell_2$-subspace embedding.

\begin{definition}[Subspace Embedding for Vector Spaces \cite{Balabanov22}]
\label{def:vec_embedding}
Let $\varepsilon \in (0, 1)$, and let $\Theta \in \mathbb{R}^{\ell \times n}$ with $\ell \ll n$ be a sketching matrix. The matrix $\Theta$ is said to be an $\varepsilon$-subspace embedding for a subspace $\mathcal{V} \subset \mathbb{R}^n$ if, for all $x \in \mathcal{V}$,
\begin{equation}
(1 - \varepsilon) \|x\|_2^2 \le \|\Theta x\|_2^2 \le (1 + \varepsilon) \|x\|_2^2,
\end{equation}
or equivalently,
\begin{equation}
(1 + \varepsilon)^{-1} \|\Theta x\|_2^2 \le \|x\|_2^2 \le (1 - \varepsilon)^{-1} \|\Theta x\|_2^2.
\end{equation}
\end{definition}

\begin{definition}[Matrix Subspace Embedding]
\label{def:matrix_embedding}
Let $V \in \mathbb{R}^{n \times k}$ have full column rank such that $\operatorname{range}(V) = \mathcal{V} \subset \mathbb{R}^n$, and let $\mathcal{U} \subset \mathbb{R}^{n \times s}$ be the matrix subspace defined by
\begin{equation*}
\mathcal{U} = \left\{ V X \in \mathbb{R}^{n \times s} : X \in \mathbb{R}^{k \times s} \right\}.
\end{equation*}
A matrix $\Theta \in \mathbb{R}^{\ell \times n}$ is called an $\varepsilon$-subspace embedding for $\mathcal{U}$ if, for all $X \in \mathbb{R}^{k \times s}$,
\begin{equation}
\label{eq3.3x}
(1 - \varepsilon) \|V X\|_F^2 \le \|\Theta V X\|_F^2 \le (1 + \varepsilon) \|V X\|_F^2,
\end{equation}
where $\|\cdot\|_F$ denotes the Frobenius norm.
\end{definition}

\begin{proposition}[Matrix embedding from vector embedding]
\label{prop:mat_from_vec}
If $\Theta \in \mathbb{R}^{\ell \times n}$ is an $\varepsilon$-subspace embedding for the vector subspace $\operatorname{range}(V) \subset \mathbb{R}^n$, then it is also an $\varepsilon$-subspace embedding for the matrix subspace $\mathcal{U} = \{ V X : X \in \mathbb{R}^{k \times s} \} \subset \mathbb{R}^{n \times s}$.
\end{proposition}

\begin{proof}
Let $X = [x_1, x_2, \dots, x_s] \in \mathbb{R}^{k \times s}$ be arbitrary. Since $V x_j \in \operatorname{range}(V)$ for each column $j = 1, \dots, s$, Definition~\ref{def:vec_embedding} implies that
\begin{equation*}
(1 - \varepsilon) \|V x_j\|_2^2 \le \|\Theta V x_j\|_2^2 \le (1 + \varepsilon) \|V x_j\|_2^2, \quad j = 1, \dots, s.
\end{equation*}
Summing these inequalities over all $s$ columns yields
\begin{equation*}
(1 - \varepsilon) \sum_{j=1}^s \|V x_j\|_2^2 \le \sum_{j=1}^s \|\Theta V x_j\|_2^2 \le (1 + \varepsilon) \sum_{j=1}^s \|V x_j\|_2^2.
\end{equation*}
Using the identity $\|M\|_F^2 = \sum_{j=1}^s \|m_j\|_2^2$ for any matrix $M = [m_1, \dots, m_s]$, we obtain
\begin{equation*}
(1 - \varepsilon) \|V X\|_F^2 \le \|\Theta V X\|_F^2 \le (1 + \varepsilon) \|V X\|_F^2.
\end{equation*}
Since $X$ is arbitrary, $\Theta$ is an $\varepsilon$-subspace embedding for $\mathcal{U}$.
\end{proof}

\begin{proposition}[Singular-value bounds for matrix subspace embeddings]
If $\Theta \in \mathbb{R}^{\ell \times n}$ is an $\varepsilon$-subspace embedding for the matrix subspace $\mathcal{U}$, then for any matrix $X \in \mathbb{R}^{k \times s}$,
\begin{equation}
(1 + \varepsilon)^{-1/2} \sigma_{\min}(\Theta V X) \le \sigma_{\min}(V X) \le \sigma_{\max}(V X) \le (1 - \varepsilon)^{-1/2} \sigma_{\max}(\Theta V X),
\label{eq3.5}
\end{equation}
where $\sigma_{\min}(\cdot)$ and $\sigma_{\max}(\cdot)$ denote the minimal and maximal singular values of the matrix, respectively.
\end{proposition}
%%%%%%%%%%%%%%%%%
\begin{proof}
By the variational characterization (Courant--Fischer min-max theorem) of singular values, for any matrix $M \in \mathbb{R}^{n \times s}$, the maximal and minimal singular values are given by
\begin{equation*}
\sigma_{\max}(M) = \max_{\|u\|_2 = 1} \|M u\|_2 \quad \text{and} \quad \sigma_{\min}(M) = \min_{\|u\|_2 = 1} \|M u\|_2.
\end{equation*}
Let $u \in \mathbb{R}^s$ with $\|u\|_2 = 1$ be arbitrary, and define the vector $w = V X u \in \mathbb{R}^n$. Since the columns of $V$ span $\operatorname{range}(V)$, we have $w \in \operatorname{range}(V)$. 

By the $\varepsilon$-subspace embedding property of $\Theta$ on $\operatorname{range}(V)$, it follows that
\begin{equation*}
(1+\varepsilon)^{-1} \|\Theta w\|_2^2 \le \|w\|_2^2 \le (1-\varepsilon)^{-1} \|\Theta w\|_2^2,
\end{equation*}
which, after taking the square root, is equivalent to
\begin{equation*}
(1+\varepsilon)^{-1/2} \|\Theta V X u\|_2 \le \|V X u\|_2 \le (1-\varepsilon)^{-1/2} \|\Theta V X u\|_2.
\end{equation*}
Taking the maximum over all unit vectors $\|u\|_2 = 1$ in the right-hand inequality yields
\begin{equation*}
\sigma_{\max}(V X) = \max_{\|u\|_2 = 1} \|V X u\|_2 \le (1 - \varepsilon)^{-1/2} \max_{\|u\|_2 = 1} \|\Theta V X u\|_2 = (1 - \varepsilon)^{-1/2} \sigma_{\max}(\Theta V X).
\end{equation*}
Similarly, taking the minimum over all unit vectors $\|u\|_2 = 1$ in the left-hand inequality gives
\begin{equation*}
\sigma_{\min}(V X) = \min_{\|u\|_2 = 1} \|V X u\|_2 \ge (1 + \varepsilon)^{-1/2} \min_{\|u\|_2 = 1} \|\Theta V X u\|_2 = (1 + \varepsilon)^{-1/2} \sigma_{\min}(\Theta V X).
\end{equation*}
Since $\sigma_{\min}(V X) \le \sigma_{\max}(V X)$ always holds, combining the two inequalities establishes \eqref{eq3.5}.
\end{proof}

%……%%%%%%%%%%%%%
\subsection{Randomized Global Arnoldi (RGl-Arnoldi) method }
\label{sub:rGMRES}
 At the $k$th iteration of the Gl-GMRES method, the solution to (\ref{MRHS}) can be approximated by:
\begin{equation}
\label{eq:GMRES_projectedproblem}
X_k = X_0 + {\mathcal{V}}_{k} \diamond {z}_k, \quad \mbox{where} \quad {z}_k = \operatorname*{argmin}_{z \in \mathbb{R}^k} \|\beta e^{(k+1)}_1-\bar{H}_k z \|_2.
\end{equation}
Note that the orthonormalization process can be costly for very large dimensional blocks.  
One remedy to mitigate the computational cost of the global modified Gram-Schmidt (Gl-Gram-Schmidt) process is to use the randomized Gl-Gram-Schmidt process that exploits sketching, as described in \cite{Balabanov2025}.

Before defining the RGl-Arnoldi, we will introduce some symbols and notations that will be used for this method.
\begin{definition}
Let $\mathcal{X}$ and $\mathcal{Y}$ be two matrices of dimension $n\times s$, we consider the  sketched inner product
\begin{eqnarray}
\label{SP}
<\mathcal{X},\mathcal{Y}>_{\Theta}=\mathrm{Trace}(\mathcal{X}^{T}\Theta^{T}\Theta\mathcal{Y}).
\end{eqnarray}
\end{definition}
The RGl-Arnoldi approach is obtained by applying the randomized Gl-Gram-Schmidt algorithm to the global Krylov subspace $\mathcal{K}_k({A},{R}_0)$. More precisely, after $k$ iterations of RGl-Arnoldi process, we have a basis $\{ Q_1,\dots,{Q}_{k+1}\}$ of $\mathcal{K}_{k+1}({A}, {R}_0)$ 
with blocks ${Q}_i$, $i=1,\dots,k+1$, that are orthonormal with respect to the sketched inner product (\ref{SP}) , and an upper Hessenberg matrix $\bar{E}_k \in \mathbb{R}^{(k+1) \times k}$. Let %\linebreak[4]
$\mathcal{Q}_{k}=[{Q}_1,\dots,{Q}_{k}]\in\mathbb{R}^{n\times ks}$ and $\mathcal{Q}_{k+1}=
[\mathcal{Q}_k,{Q}_{k+1}]\in\mathbb{R}^{n\times (k+1)s}$, then we get a RGl-Arnoldi process identity analogous to Gl-Arnoldi method (see Section \ref{Gl-Arnoli})

\begin{equation}
\label{eq:rArnoldi} {A}  \mathcal{Q}_{k} = \mathcal{Q}_{k+1}\diamond \bar{E}_k.
\end{equation}
Algorithm \ref{algo1x} summarizes the main steps involved in the new RGl-Arnoldi process:
\begin{algorithm}[H]
\caption{The RGl-Arnoldi process}
\begin{algorithmic}[1]
\STATE  Initialize
$\tilde{Q}_1 = R_0$.
\STATE Sketch $\tilde{S}_{1}=\Theta R_{0}$.
\STATE Compute the sketched (semi) norm $\beta=\|\tilde{S}_{1}\|_{F}$.
\STATE Scale blocks
$S_{1}=\tilde{S}_{1}/\beta$ and $Q_{1}=\tilde{Q}_{1}/\beta$.
\FOR{$j = 1,2,\ldots,k$}
  \STATE $W := AQ_j$
   \STATE $Z :=\Theta  W$
  \FOR{$i = 1,2,\ldots,j$}
    \STATE $E_{i,j} = \langle W, Q_i \rangle_{\Theta}=\langle Z, S_i \rangle_{F}$
    \STATE $W = W - E_{i,j} Q_i$
      \STATE $Z = Z - E_{i,j} S_i$
  \ENDFOR
  \STATE $E_{j+1,j} = \|Z\|_F$
  \STATE $Q_{j+1} =W/E_{j+1,j}$ and  $S_{j+1} =Z/E_{j+1,j}$.
\ENDFOR
\end{algorithmic}
\label{algo1x}
\end{algorithm}

 Building on the RGl-Arnoldi process, the RGl-GMRES solution at the $k$th iteration is given by:
\begin{equation}
\label{eq:rGMRES_projectedproblem}
X_k^{r} = X^{r}_0 + \mathcal{Q}_k\diamond  z_k^{r}\,
\quad where \quad {z}_k^{r}= \text{argmin}_{z \in \mathbb{R}^k}  \|\bar{E}_k {z} - \tilde{\beta} {e}_1 \|_2\quad\mbox{and}\quad \tilde{\beta}=\|{\Theta}{R}_0\|_F.
\end{equation}
The $k$th RGl-GMRES approximation $X_k^{r}$  minimizes the sketched (semi)norm of the residual $R^{r}_k$ over $\calK_k(A,R_0)$, since
\begin{equation}
\label{eq:GMRES_sketchedprojected}
\begin{split}
\|\Theta ({A}{X}_k^{r}- {B})\|_{F} &=\min_{z \in \mathbb{R}^k}  \|{\Theta} ({A}  \mathcal{Q}_k \diamond{z} -  R_0) \|_F\\
& = \min_{z \in \mathbb{R}^k}  \|{\Theta}\mathcal{Q}_{k+1} \diamond( {z} - \tilde{\beta} {e}_1 ) \|_{F}\\
&=\min_{z \in \mathbb{R}^k} \|\bar{E}_k{{z}} - \tilde{\beta} {e}_1 \|_2.
\end{split}
\end{equation}
For solving the problem \eqref{eq:GMRES_sketchedprojected}, the matrix \(\bar{E}_k\) is factored as \(\bar{Q}_k \bar{E}_k = \bar{R}_k\), where \(\bar{Q}_k\) is a product of Givens rotations and \(\bar{R}_k\) is upper triangular. In practice, the QR-factorization is implemented at each step by performing this factorization on the new column of \(\bar{E}_k\); the details of this implementation are given in \cite{Saad1986}. Similarly, this leads to considerable savings compared to the randomized Bl-GMRES method where the Hessenberg matrix $E_k$ is of dimension $ks\times ks$
or to the simple randomized GMRES algorithm where $s$ least-squares problems have to be solved at each
iteration. 

Based on the above analysis, we now present the RGl-GMRES algorithm for large-scale linear systems with multiple right-hand sides \eqref{MRHS}:
\begin{algorithm}[t]
%\caption{ The randomized global GMRES (rgGMRES) algorithm}
\begin{algorithmic}[1]
 \STATE\textbf{Require: } 
\[
A \in \mathbb{R}^{n \times n}, \quad
B \in \mathbb{R}^{n\times s}, \quad
\Theta \in \mathbb{R}^{\ell \times n}, 
\quad \ell \ll n
\]
\STATE  Initialize
$\tilde{V}_1 = R_0$.
\STATE Sketch $\tilde{S}_{1}=\Theta R_{0}$.
\STATE Compute the sketched (semi) norm $\tilde{\beta}=\|\tilde{S}_{1}\|_F$.
\STATE Scale blocks
$V_{1}=\tilde{V}_{1}/\tilde{\beta}$ and $S_{1}=\tilde{S}_{1}/\tilde{\beta}$.
\FOR{$j = 1,2,\ldots,k$}
  \STATE $W := A V_j$.
  \STATE Sketch $Z=\Theta W$.
  \FOR{$i = 1,2,\ldots,j$}
    \STATE $h_{i,j} = \langle W, V_i \rangle_{\Theta}=\langle Z, S_i \rangle_{F}$.
    \STATE $W = W - h_{i,j} V_i$.
    \STATE $Z = Z - h_{i,j} S_i$.
  \ENDFOR
 \STATE Compute the sketched (semi)norm 
   $h_{j+1,j} = \|Z\|_F$.
  \STATE Scale blocks:  $V_{j+1} =W/h_{j+1,j}$ and  $S_{j+1} =Z/h_{j+1,j}$.
\ENDFOR
\STATE Solve the sketched least square  problem:
\[
\min_{z\in\mathbb{R}^{k}} \|\bar{H}_kz - \tilde{\beta}e_1\|_2,
\]
\STATE Set
\[
X_k = X_0 + V_k \diamond z,
\qquad
R_k = B - A X_k.
\]
\caption{The RGl-GMRES algorithm}
\label{alg:rGMRES}
\end{algorithmic}
\end{algorithm}

Similar to the Gl-GMRES method, the work and storage of our RGl-GMRES method increase with the iteration step. This may cause
difficulties for large steps. A frequently-used way is to restart our RGl-GMRES algorithm every $m$ iterations
many times as necessary for convergence. In practice, this is usually done for the GMRES algorithm \cite{Saad1986}. 

Next, we can present some propositions about the proposed RGl-GMRES method.

\begin{definition}
\label{def4x}
This optimality property of RGl-GMRES can be also derived from more general facts about projection methods 
%In the general framework for iterative projection methods for \eqref{eq:LS} 
(as described, for instance, in \cite{Saad1986}). Indeed, 
% GMRES is such that $\bfx_k\in\calK_k(\bfA,\bfb)$ and $\bfr_k\perp\bfA\calK_k(\bfA,\bfb)$. Similarly, 
the $k$-th RGl-GMRES approximation is such that 
\[
 X_{k}^{r}\in\calK_k(A, R_0)\quad\mbox{and}\quad  R^{r}_k\perp_{<\Theta,\Theta>_{F}}A\calK_k(A,R_0),
\]
where, by $\perp_{{\Theta}}$, we mean orthogonal with respect to the sketched inner product $\langle\Theta\cdot,\Theta\cdot\rangle_{F}$.
\end{definition}
Note that the above properties generalize to those used by standard Gl-GMRES, which would hold with $\Theta=I_n$. 
\begin{proposition}
\label{pro5}
Assume that $\Theta$ is an $(\epsilon,\delta,K+1)$-oblivious embedding, then we have  the following bound for the RGl-GMRES residual $B-AX_k$ norm with respect to the Gl-GMRES residual $ B-AX_k$ norm at the $k$th iteration
    \[
\| B-A {X}_k^{r}\|^2_{F}\leq \frac{(1+\epsilon)}{(1-\epsilon)}\| B-A{ X}_k\|^2_{F}.
\]
\end{proposition}
\begin{proof} The $k$-th Gl-GMRES iterate $X_k$ is defined as:
\begin{eqnarray} 
\label{opt1}
 X_k
\;=\;
\arg\min_{X \in X_0 + \mathcal{K}_k(A,R_0)}
\| B - A X \|_F .
\end{eqnarray}
That is, Gl-GMRES method  minimizes the true residual norm over the affine Krylov subspace.
The $k$-th RGl-GMRES iterate $X_k^{r}$ is defined as :
\begin{eqnarray}
\label{opt2}
X_k^{r}
\;=\;
\arg\min_{X \in X_0 + \mathcal{K}_k(A,R_0)}
\| \Theta ( B - A X ) \|_F .
\end{eqnarray}
That is, the RGl-GMRES minimizes the sketched residual norm over the same affine Krylov subspace. Since $\Theta$ is an $(\varepsilon,\delta,k+1)$-subspace embedding and by using Definition \ref{def:matrix_embedding}, we get
\begin{eqnarray*}
\label{eq2}
(1-\epsilon)\| B-AX_k^{r}\|^{2}_F\leq
\|\Theta(B-A X_k^{r})\|^2_{F},
\end{eqnarray*}
using the RGl-GMRES optimality (\ref{opt2}) and inequality \eqref{eq3.3x}
\begin{eqnarray*}
\label{eq2x}
\|\Theta({B}-A X_k^{r})\|^2_{F}
\leq
\|\Theta({B}-A X_k)\|^2_{F}\leq
(1+\epsilon)\| B-A {X}_k\|^2_{F}.
\end{eqnarray*}
This completes the proof.
\end{proof}

Under the same assumptions on the coefficient matrix $A$ as in Section~\ref{sec2.2}, we now present the main convergence theorem for our proposed RGl-GMRES method.

\begin{theorem}
\label{thm:rgl_convergence}
Let $A = Z \Lambda Z^{-1}$ be diagonalizable, and let $R_k^r = B - AX_k^r$ be the $k$-th residual obtained by the RGl-GMRES method applied to~\eqref{MRHS}. Let $\Gamma_\epsilon = \frac{1+\epsilon}{1-\epsilon}$. 
Then the residual $R_k^r$ satisfies:
\begin{equation}
\label{eq:rgl_bound_gamma}
\frac{\|R_k^r\|_F^2}{\|\beta\|_F^2} \le \Gamma_\epsilon \frac{\|Z\|_2^2}{e_1^T (\overline{V}_{k+1}^H D_\gamma \overline{V}_{k+1})^{-1} e_1},
\end{equation}
and consequently,
\begin{equation}
\label{eq:rgl_minmax_bound}
\frac{\|R_k^r\|_F}{\|R_0\|_F} \le \sqrt{\Gamma_\epsilon} \, \kappa_2(Z) \min_{p \in \mathcal{P}_k, p(0)=1} \max_{\lambda \in \operatorname{Sp}(A)} |p(\lambda)|,
\end{equation}
where $D_\gamma$, $\overline{V}_{k+1}$, and $\mathcal{P}_k$ are as defined in Section~\ref{sec2.2}.
\end{theorem}

\begin{proof}
By Proposition~\ref{pro5} and Theorem~\ref{thm:gl_convergence}, the residual $R_k^r$ of RGl-GMRES and the residual $R_k$ of standard Gl-GMRES satisfy
\begin{equation}
\label{eq:proof_step1}
\|R_k^r\|_F^2 \le \Gamma_\epsilon \|R_k\|_F^2 \le \Gamma_\epsilon \frac{\|\beta\|_F^2 \|Z\|_2^2}{e_1^T (\overline{V}_{k+1}^H D_\gamma \overline{V}_{k+1})^{-1} e_1},
\end{equation}
which immediately establishes~\eqref{eq:rgl_bound_gamma}.

To derive the min-max bound~\eqref{eq:rgl_minmax_bound}, we note that $R_0 = Z\beta$ implies $\beta = Z^{-1}R_0$, and hence $\|\beta\|_F \le \|Z^{-1}\|_2 \|R_0\|_F$. Substituting this relation into~\eqref{eq:proof_step1} yields
\begin{equation}
\frac{\|R_k^r\|_F^2}{\|R_0\|_F^2} \le \Gamma_\epsilon \frac{\|Z\|_2^2 \|Z^{-1}\|_2^2}{e_1^T (\overline{V}_{k+1}^H D_\gamma \overline{V}_{k+1})^{-1} e_1} = \Gamma_\epsilon \frac{\kappa_2(Z)^2}{e_1^T (\overline{V}_{k+1}^H D_\gamma \overline{V}_{k+1})^{-1} e_1}.
\label{eq:rel_res_mid}
\end{equation}
Next, invoking the equivalence between the Krylov--Vandermonde quadratic form maximization and the discrete min-max polynomial approximation (see \cite{bellalij2008analysis}):
% \begin{equation*}
% \frac{1}{e_1^T (\overline{V}_{k+1}^H D_\gamma \overline{V}_{k+1})^{-1} e_1} \le \max_{\gamma \in \widetilde{\Delta}_n} \frac{1}{e_1^T (\overline{V}_{k+1}^H D_\gamma \overline{V}_{k+1})^{-1} e_1} = \left( \min_{p \in \mathcal{P}_k, p(0)=1} \max_{\lambda \in \operatorname{Sp}(A)} |p(\lambda)| \right)^2.
% \end{equation*}
\begin{equation*}
\frac{1}{e_1^T (\overline{V}_{k+1}^H D_\gamma \overline{V}_{k+1})^{-1} e_1} 
\le \max_{\substack{\gamma \ge 0 \\ \sum_{i=1}^n \gamma_i = 1}} \frac{1}{e_1^T (\overline{V}_{k+1}^H D_\gamma \overline{V}_{k+1})^{-1} e_1} 
= \left( \min_{p \in \mathcal{P}_k, p(0)=1} \max_{\lambda \in \operatorname{Sp}(A)} |p(\lambda)| \right)^2.
\end{equation*}
Substituting this into~(\ref{eq:rel_res_mid}) and taking the square root on both sides completes the proof.
\end{proof}
%%%%%%%%%
\begin{remark}
\label{rem:convergence_insights}
Several important observations follow from Theorem~\ref{thm:rgl_convergence} in comparison with Theorem~\ref{thm:gl_convergence}:
\begin{enumerate}[1)]
    \item \textbf{Asymptotic Consistency:} In the limit as $\epsilon \to 0$ (exact inner products), we have $\Gamma_\epsilon = \frac{1+\epsilon}{1-\epsilon} \to 1$. Thus, the upper bounds~\eqref{eq:rgl_bound_gamma} and~\eqref{eq:rgl_minmax_bound} reduce exactly to the deterministic Gl-GMRES results in Theorem~\ref{thm:gl_convergence}.
    \item \textbf{Negligible Impact on Convergence Rate:} The presence of random sketching only introduces a moderate constant inflation factor $\sqrt{\Gamma_\epsilon}$ to the residual bound, while the asymptotic convergence rate remains entirely governed by $\kappa_2(Z)$ and the spectrum $\operatorname{Sp}(A)$. For instance, with $\epsilon = 0.1$, $\sqrt{\Gamma_\epsilon} \approx 1.105$, which explains why RGl-GMRES requires essentially the same number of iterations as Gl-GMRES in practice.
    \item \textbf{Special Cases:} When $A$ is normal, $\kappa_2(Z) = 1$. Furthermore, for $s = 1$ and $\epsilon \to 0$, the bound~\eqref{eq:rgl_minmax_bound} recovers the classic single-system GMRES bound \cite{Saad1986}.
\end{enumerate}
\end{remark}

\subsection{The equivalence between RGl-GMRES and classical RGMRES Methods}
System~(\ref{MRHS}) can be equivalently rewritten as the following $n s \times n s$ block-diagonal linear system:
\begin{equation}
\label{eq:vec_sys}
\mathcal{A} \mathcal{X} = \mathcal{B},
\end{equation}
where $\mathcal{A} = I_s \otimes A \in \mathbb{R}^{ns \times ns}$, $\mathcal{X} = \operatorname{vec}(X) \in \mathbb{R}^{ns}$, and $\mathcal{B} = \operatorname{vec}(B) \in \mathbb{R}^{ns}$.

To establish the connection between RGl-GMRES applied to~\eqref{MRHS} and classical randomized GMRES (RGMRES) applied to~(\ref{eq:vec_sys}), we define the structured sketching matrix associated with $\Theta \in \mathbb{R}^{\ell \times n}$ as
\begin{equation}
\mathbf{\Theta} = I_s \otimes \Theta \in \mathbb{R}^{\ell s \times ns}.
\end{equation}
For any $X, Y \in \mathbb{R}^{n \times s}$, using the identity $\operatorname{vec}(\Theta X) = (I_s \otimes \Theta)\operatorname{vec}(X) = \mathbf{\Theta} \operatorname{vec}(X)$, the sketched Frobenius inner product satisfies
\begin{equation*}
\begin{split}
\langle X, Y \rangle_\Theta & = \operatorname{Trace}(X^T \Theta^T \Theta Y) \\
%& = \operatorname{vec}(\Theta X)^T \operatorname{vec}(\Theta Y) \\
& = \langle \mathbf{\Theta}\operatorname{vec}(X), \mathbf{\Theta}\operatorname{vec}(Y) \rangle_2 \\
& = \langle \operatorname{vec}(X), \operatorname{vec}(Y) \rangle_{\mathbf{\Theta}}.
\end{split}
\end{equation*}
Based on this relation, we have the following equivalence theorem.
\begin{theorem}
\label{thm:equivalence}
Let $X_0 = [x_0^{(1)}, x_0^{(2)}, \dots, x_0^{(s)}] \in \mathbb{R}^{n \times s}$ be an initial approximation for the multiple right-hand-side system~(1.1), and let $\mathcal{X}_0 = \operatorname{vec}(X_0) \in \mathbb{R}^{ns}$ be the initial approximation for the vectorized linear system~(\ref{eq:vec_sys}). 
Let $\Theta \in \mathbb{R}^{\ell \times n}$ be the sketching matrix used in RGl-GMRES, and suppose the classical RGMRES applied to~(\ref{eq:vec_sys}) employs the structured sketching matrix $\mathbf{\Theta} = I_s \otimes \Theta \in \mathbb{R}^{\ell s \times ns}$. 

Then, at each step $k$, the iterate $X_k^r$ generated by RGl-GMRES and the iterate $\mathcal{X}_k^r$ generated by RGMRES satisfy
\begin{equation}
\operatorname{vec}(X_k^r) = \mathcal{X}_k^r.
\end{equation}
\end{theorem}
%%%%%%%
\begin{proof}
First, observe that for any $j \ge 0$,
\begin{equation*}
\operatorname{vec}(A^j R_0) = (I_s \otimes A^j) \operatorname{vec}(R_0) = \mathcal{A}^j \mathcal{R}_0,
\end{equation*}
where $\mathcal{R}_0 = \operatorname{vec}(R_0) = \mathcal{B} - \mathcal{A}\mathcal{X}_0$. Hence, the vectorization of the global Krylov subspace $\mathcal{K}_k(A, R_0)$ coincides with the standard Krylov subspace $\mathcal{K}_k(\mathcal{A}, \mathcal{R}_0)$, i.e.,
\begin{equation*}
\operatorname{vec}(\mathcal{K}_k(A, R_0)) = \mathcal{K}_k(\mathcal{A}, \mathcal{R}_0).
\end{equation*}
Second, by definition, the RGl-GMRES iterate $X_k^r \in X_0 + \mathcal{K}_k(A, R_0)$ uniquely solves the sketched least-squares problem:
\begin{equation*}
X_k^r = \arg\min_{X \in X_0 + \mathcal{K}_k(A, R_0)} \|\Theta(B - AX)\|_F^2.
\end{equation*}
Using the vectorized identity $\|\Theta M\|_F^2 = \|\mathbf{\Theta}\operatorname{vec}(M)\|_2^2$ with $M = B - AX$, this optimization problem is equivalent to
\begin{equation*}
\operatorname{vec}(X_k^r) = \arg\min_{\mathcal{X} \in \mathcal{X}_0 + \mathcal{K}_k(\mathcal{A}, \mathcal{R}_0)} \|\mathbf{\Theta}(\mathcal{B} - \mathcal{A}\mathcal{X})\|_2^2,
\end{equation*}
which is precisely the defining minimization problem of the $k$-th iterate $\mathcal{X}_k^r$ of classical RGMRES applied to~(\ref{eq:vec_sys}) with sketching matrix $\mathbf{\Theta}$. By uniqueness of the least-squares solution, $\operatorname{vec}(X_k^r) = \mathcal{X}_k^r$.
\end{proof}
\section{Application to the Lyapunov Matrix Equation}
\label{sec4x}
In this section, we discuss how the proposed RGl-GMRES framework can be straightforwardly extended to solve the continuous-time Lyapunov matrix equation:
\begin{equation}
\label{eq:lyapunov}
AX + XA^T + BB^T = 0,
\end{equation}
where $A, X \in \mathbb{R}^{n \times n}$ and $B \in \mathbb{R}^{n \times s}$ with $s \ll n$. Lyapunov matrix equations play an important role in control theory, particularly in stability analysis and model order reduction. Equation~\eqref{eq:lyapunov} admits a unique solution if and only if the spectra of $A$ and $-A$ are disjoint.

A straightforward approach to solving~\eqref{eq:lyapunov} is to transform it into an equivalent linear system via vectorization:
\begin{equation}
\label{eq:vec_lyap}
\widetilde{\mathcal{A}} \widetilde{x} = \widetilde{b},
\end{equation}
where $\widetilde{\mathcal{A}} = I \otimes A + A \otimes I \in \mathbb{R}^{n^2 \times n^2}$, $\widetilde{x} = \operatorname{vec}(X) \in \mathbb{R}^{n^2}$, and $\widetilde{b} = -\operatorname{vec}(BB^T) \in \mathbb{R}^{n^2}$.

Alternatively, we can define the linear matrix operator
\begin{equation*}
\mathcal{A} : \mathbb{R}^{n \times n} \to \mathbb{R}^{n \times n}, \quad \mathcal{A}(X) = AX + XA^T.
\end{equation*}
Then, the Lyapunov matrix equation~\eqref{eq:lyapunov} can be expressed in operator form as
\begin{equation}
\label{eq:lyap_operator}
\mathcal{A}(X) = -BB^T.
\end{equation}
The RGl-GMRES method \cite{JBILOU1999} can be naturally applied to solve~\eqref{eq:lyap_operator} without explicitly forming the prohibitively large Kronecker-product matrix $\widetilde{\mathcal{A}}$. More specifically, the method is implemented by replacing each matrix-vector product with the action of the Lyapunov operator $\mathcal{A}(X) = AX + XA^T$.

Starting from an initial approximation $X_0 \in \mathbb{R}^{n \times n}$ with initial residual
\begin{equation*}
R_0 = -BB^T - \mathcal{A}(X_0) = -BB^T - AX_0 - X_0 A^T,
\end{equation*}
the RGl-GMRES method constructs the $k$-th approximate solution $X_k \in X_0 + \mathcal{K}_k(\mathcal{A}, R_0)$ by minimizing the sketched Frobenius norm of the residual:
\begin{equation*}
R_k = -BB^T - AX_k - X_k A^T.
\end{equation*}
This formulation has an important computational advantage. Indeed, although the vectorized formulation \eqref{eq:vec_lyap} involves the $n^2\times n^2$ matrix $\widetilde{\mathcal{A}}$, the RGl-GMRES approach operates directly on $n\times n$ matrices and only requires the evaluation of the operator $\mathcal{A}$. Consequently, the Kronecker-product matrix does not need to be explicitly constructed or stored, which can substantially reduce both the memory requirements and the computational cost when $n$ is large. The resulting procedure is summarized in Algorithm~\ref{alg:glgmres-lyapunov}.
%
%
%Algorithm \ref{alg:rGMRES}  can be reformulated for the solution  of the Lyapunov matrix equation \eqref{eq:lyapunov} as follows:
\begin{algorithm}[t]
\caption{RGl-GMRES for solving the Lyapunov matrix equation \eqref{eq:lyapunov}}
\label{alg:glgmres-lyapunov}
\begin{algorithmic}[1]
\STATE Choose an initial approximation $X_0$.
\STATE Compute the initial residual
\[
R_0=-BB^T-AX_0-X_0A^{\mathrm{T}},
\]
% and set
% \[
% V_1=\frac{R_0}{\|R_0\|_F}.
% \]
%
\FOR{$j=1,2,\ldots,k$}
    \STATE Apply Algorithm \ref{algo1x} with the operator
    \[
    \mathcal{A}(X)=AX+XA^{\mathrm{T}}
    \]
    to construct the sketched-orthonormal basis
    $V_1,V_2,\ldots,V_{k+1}$ of the matrix Krylov subspace
    \[
    \mathcal{K}_k(\mathcal{A},R_0)
    =
    \operatorname{span}
    \left\{
    R_0,\mathcal{A}(R_0),\ldots,\mathcal{A}^{k-1}(R_0)
    \right\}.
    \]
\ENDFOR
\STATE Solve the sketched least-squares problem
\[
y_k=
\underset{y\in\mathbb{R}^{k}}{\operatorname{arg\,min}}
\left\|
\|\Theta R_0\|_F e_1^{(k+1)}
-\bar{H}_k y
\right\|_2.
\]
\STATE Compute the approximate solution
\[
X_k = X_0 + \mathcal{V}_k\diamond y_k,\quad R_k = -BB^T - \mathcal{A}(X_k),
\]
where
\[
\mathcal{V}_k y_k
=
\sum_{i=1}^{k}y_k^{(i)}V_i.
\]
\STATE \textbf{Return} $X_k$.
\end{algorithmic}
\label{algo3}
\end{algorithm}
%\newpage
\section{Numerical Results}
\label{sec4}
In this section, we present numerical experiments to evaluate the performance of the proposed methods in solving \eqref{MRHS} as well as the Lyapunov matrix equation, and we compare their effectiveness with that of the other methods considered in this paper. Here we only consider Gl-GMRES and RGl-GMRES methods mentioned in Algorithm $2$ and Algorithm \ref{algo3},
respectively. 
For the randomized algorithms, several sketch sizes \(\ell\) and distributions for the sketching matrices \(\Theta\) are considered. In Step 3 of Algorithm 2, we employ Gaussian sketching. No significant difference in performance, in terms of stability or approximation accuracy for a fixed sketch size \(\ell\), was observed between Rademacher (or Gaussian) sketching and SRHT sketching. Therefore, in this section, we present only the results obtained using Gaussian sketching.
%For the first experiment, we consider the matrix \texttt{fs\_680\_1}, obtained from the
%SuiteSparse Matrix Collection (\url{https://sparse.tamu.edu}, accessed on 1 January 2025).
%The matrix is scaled to have a unit diagonal and is denoted by \texttt{fs\_680\_1c}.
%This sparse matrix has order $680$, contains $21{,}184$ nonzero entries, and has a
%condition number equal to $8.6944 
\subsection{Solving multiple linear systems}
The initial matrix $X_{0}$ is set to be a random matrix with suitable size. When solving linear systems with multiple right-hand sides (\ref{MRHS}) in Algorithms $2$ and $3$, the iterative process is terminated when the relative residual norm
\[
\frac{\|B-AX_{k}\|_{F}}{\|B-AX_0\|_{F}}
\]
falls below the prescribed tolerance $\tau=10^{-12}$, where $X_{k}$ denotes the approximate solution of (\ref{MRHS}). In the sense of the following terms:
\begin{itemize}
\item $\ell$: denotes the sketch size.
\item CPU: denotes the elapsed CPU time of the convergence in (seconds).
\item $\mathrm{Iter}$: the number of iterations for the Gl-GMRES method or the randomized Gl-GMRES method, depending on which one is used.
%\item  $\mathrm{Iter}_{\mathcal{P}Gl-GMRES}$: number of iterations of preconditioned Gl-GMRES method _{Gl-GMRES}
%\item  $\mathrm{Iter}_{RGl-GMRES}$: number of iterations of   
%\item  $\mathrm{Iter}_{\mathcal{P}RGl-GMRES}$: number of iterations of preconditioned  randomized\\ Gl-GMRES method,
\item $\mathrm{Res}$: norm of absolute  residual is defined as
\begin{equation*}
\mathrm{Res}=\|{B}-{A}X_{k}\|_{F}.
%\nonumber
\end{equation*}
%\item $\mathrm{RES}$: norm of absolute error is given as follows:
%\begin{eqnarray}
%\mathrm{ERR}=\|X^{*}-\mathrm{X}^{(k)}\|_{F}, \nonumber
%\end{eqnarray}
\end{itemize}
In addition, the right-hand sides are generated using randomly selected solution vectors. 
% All computations were performed on a computer with an Intel Core™ i7-870H CPU @ 2.93 GHz and 16 GB of RAM, using MATLAB R2017a. 
In Table~1, we present the number of iterations (Iter), CPU time, and residual norm (Res) obtained by the Gl-GMRES and RGl-GMRES methods for a range of problem sizes and different values of the sketch size $\ell$.
\begin{exe}
\label{ex1}
We consider the steady incompressible Navier--Stokes equations on the
square domain
\[
\Omega_{\square}=(-1,1)\times(-1,1).
\]
The computational configuration corresponds to a two-dimensional cavity
driven by the motion of its upper boundary. In particular, the lid moves
horizontally from left to right with prescribed velocity
$u_x=1$, while homogeneous Dirichlet conditions are imposed on the
remaining three sides of the cavity. This benchmark problem is commonly
used for assessing numerical methods for incompressible flow problems;
see, for example, \cite{Elman2007}.

For the spatial discretization, we employ the mixed
$\mathrm{Q}_{2}-\mathrm{P}_{1}$ finite-element formulation available in
the IFISS package \cite{Elman2007}. The velocity field is approximated
using biquadratic $\mathrm{Q}_{2}$ finite elements, whereas the pressure
is represented by piecewise linear $\mathrm{P}_{1}$ finite elements.
Thus, the resulting approximation satisfies the standard mixed
finite-element framework for the velocity--pressure formulation of the
Navier--Stokes equations.
The local arrangement of the velocity and pressure degrees of freedom
associated with a $\mathrm{Q}_{2}-\mathrm{P}_{1}$ element is depicted in
Figure~\ref{nodal}. The black markers correspond to the velocity degrees
of freedom, while the white markers represent the pressure degrees of
freedom. The axes $(\mathcal{E}_{1},\mathcal{E}_{2})$ indicate the local
coordinate system used for the element.
\begin{figure}[H]
\tikzstyle{quadri}=[circle,draw,fill=black,text=white]
\tikzstyle{quadri2}=[circle,draw,fill=white,text=black]
\begin{center}
\begin{tikzpicture}
	\draw[->] (0,0) -- (3,0) node[right] {$\mathcal{E}_{1}$};
	\draw[->] (0,0) -- (0,3) node[above] {$\mathcal{E}_{2}$};
\end{tikzpicture}
\hspace{1cm}
\begin{tikzpicture}
\node[quadri] (E) at(0,4) {};
\node[quadri] (A) at(-2,4) {};
\node[quadri] (Z) at(-4,4) {};
\node[quadri] (W) at(-4,1){};
\node[quadri] (P) at(-2,1){};
\node[quadri] (O) at(0,1){};
\node[quadri2] (H) at(-2,3.2){};
\node[quadri2] (K) at(-1.1,1.8){};
\node[quadri2] (Y) at(-2.7,1.8){};
\node[quadri] (X) at(0,2.5){};
\node[quadri] (I) at(-2,2.5){};
\node[quadri] (J) at(-4,2.5){};
\draw[-,=latex] (E)--(Z);
\draw[-,=latex] (Z)--(W);
\draw[-,=latex] (W)--(O);
\draw[-,=latex] (O)--(E);
\end{tikzpicture}
\end{center}
\caption{}\text{$\mathrm{Q}_{2}-\mathrm{P}_{1}$ element  $\left(\begin{tikzpicture}\node[quadri] (P) at (0,0)  {};\end{tikzpicture}\hspace{0.1cm}\text{ velocity node}; \begin{tikzpicture}\node[quadri2] (Q) at (0,0) {};\end{tikzpicture}\hspace{0.1cm} \text{pressure node} \right),\hspace{0.1cm}\text{local co-ordinate}  \left(\mathcal{E}_{1},\mathcal{E}_{2}\right)$.}
  \label{nodal}
\end{figure}
We thereby obtain the system matrix $A$ associated with the linear systems \eqref{MRHS}. To construct a hierarchy of discretized linear systems corresponding to successive mesh refinement levels $l = 7$ and $8$, we employ the IFISS software package developed by Elman et al.~\cite{Elman2007}. This framework provides the system matrix $A$ arising from the standard lid-driven cavity benchmark problem. Table~\ref{Table1} summarizes general information about the test problems, including the dimension $n$. In the following numerical results, we set the viscosity to $\nu = 1/100$.
\begin{table}[htpb]
\begin{center}
\begin{minipage}{\textwidth}
\caption{The size of the matrix $A$}
\vspace{0.5em}
\label{tab:lid}
\begin{tabular*}{\textwidth}{@{\extracolsep{\fill}}lcccccccccc@{\extracolsep{\fill}}}
\toprule%
& \multicolumn{3}{@{}c@{}}{Lid driven cavity\footnotemark[1]}  \\\cmidrule{2-4}%
$l$ & $n$ &size of $A$  
%& %size of $N_{k}$ 
%\\
%\midrule
%$4$  & $578$  &  $578\times 578$%& $578\times 578$ 
%\\
%$5$  & $2178$   & $2178\times 2178$ 
%& $2178\times 2178$ 
%\\
%$6$ & $8450$   &  $8450\times 8450$  
\\
$7$ & $33282$   &  $33282\times 33282$  
\\
$8$ & $132098$   &  $132098\times 132098$  
%& $8450\times 8450$ 
%\botrule
\end{tabular*}
\label{Table1}
\footnotetext[1]{Lid driven cavity.}
\end{minipage}
\end{center}
\end{table}
\end{exe}
A comparison between the Gl-GMRES and RGl-GMRES methods for solving the
multiple-right-hand-side linear system~\eqref{MRHS}, with
$n=33282$, is reported in Table~\ref{Tabconvd}.
The two methods are assessed according to three performance indicators:
the computational time (CPU), the number of iterations (Iter), and the
final residual norm (Res). The notation used in the table is defined as
follows:
\begin{itemize}
    \item \textbf{Gl-GMRES}: the global GMRES method applied to the
    linear system~\eqref{MRHS}.
    
    \item \textbf{RGl-GMRES}: the randomized global GMRES method, in which
    Gaussian sketching is used to construct the randomized inner products.
    
    \item \textbf{NRHS}: the number of right-hand sides, denoted by $s$.
\end{itemize}
\begin{table}[htbp]
\centering
\caption{Performance comparison of Gl-GMRES and RGl-GMRES for different numbers of right-hand sides (Example 1).}
\vspace{0.5em}
\label{Tabconvd}
\scriptsize
\setlength{\tabcolsep}{2.2pt}
\renewcommand{\arraystretch}{1.1}
\begin{tabular}{l c|ccc|ccc|ccc|ccc}
\hline
\multirow{2}{*}{Method} & \multirow{2}{*}{$\ell$}
& \multicolumn{3}{c|}{NRHS=50}
& \multicolumn{3}{c|}{NRHS=100}
& \multicolumn{3}{c|}{NRHS=200}
& \multicolumn{3}{c}{NRHS=300} \\
& & CPU & Iter & Res.
  & CPU & Iter & Res.
  & CPU & Iter & Res.
  & CPU & Iter & Res. \\
\hline
Gl-GMRES & --
& 24.11 & 205 & $4.00\! \times\!10^{-11}$
& 34.86  & 205 & $4.20\! \times\!10^{-11}$
& 68.87  & 205 & $4.40\! \times\!10^{-11}$
& 224.80 & 205 & $4.44\! \times\!10^{-11}$ \\
RGl-GMRES & 20
& 11.11 & 203 & $1.10\! \times\!10^{-9}$
& 18.03 & 203 & $1.18\! \times\!10^{-09}$
& 39.23 & 203 & $1.23\! \times\!10^{-09}$
& 184.05 & 203 & $1.27\! \times\!10^{-09}$ \\
         & 50
& 11.62 & 204 & $4.47\! \times\!10^{-11}$
& 20.25 & 204 & $4.67\! \times\!10^{-11}$
& 40.69 & 204 & $4.85\! \times\!10^{-11}$
& 193.49 & 204 & $4.87\! \times\!10^{-11}$ \\
        & 100
& 14.43 & 204 & $4.22\! \times\!10^{-11}$
& 25.15 & 204 & $4.42\! \times\!10^{-11}$
& 42.25 & 204 & $4.61\! \times\!10^{-11}$
& 194.23 & 204 & $4.64\! \times\!10^{-11}$ \\
\hline
\end{tabular}
%\label{tab2}
\end{table}
Table~\ref{Tabconvd} shows that RGl-GMRES consistently reduces the CPU time compared with standard Gl-GMRES while requiring essentially the same number of iterations, which is fully consistent with the theoretical analysis in Remark~\ref{rem:convergence_insights}. A smaller sketch size, $\ell=20$, provides the largest computational savings but yields slightly larger residuals of order $10^{-9}$. Increasing the sketch size to $\ell=50$ or $\ell=100$ recovers residuals of order $10^{-11}$, comparable to those of standard Gl-GMRES, at a moderate additional computational cost. These results indicate that RGl-GMRES provides an effective trade-off between computational efficiency and accuracy for multiple right-hand-side systems \eqref{MRHS}. 
Table~\ref{Tabconvd2} presents a comparison of the Gl-GMRES and RGl-GMRES methods for solving the multiple-right-hand-side linear system~\eqref{MRHS} for $n=132098$.

\begin{table}[htbp] \centering \caption{Performance comparison of Gl-GMRES and RGl-GMRES for different numbers of right-hand sides (Example 1) with $n=132098$.} \vspace{0.5em} \label{Tabconvd2}
\scriptsize
\setlength{\tabcolsep}{2.2pt}
\renewcommand{\arraystretch}{1.1}
\begin{tabular}{l c|ccc|ccc|ccc|} \hline \multirow{2}{*}{Method} & \multirow{2}{*}{$\ell$} & \multicolumn{3}{c|}{NRHS=10} & \multicolumn{3}{c|}{NRHS=20} & \multicolumn{3}{c|}{NRHS=30} \\ & & CPU & Iter & Res. & CPU & Iter & Res. & CPU & Iter & Res. \\ \hline Gl-GMRES & -- & 547.17 & 4957 & $1.15\! \times\!10^{-12}$ & 940.21 & 4497 & $1.15\! \times\!10^{-12}$ & 1302.23 & 4603 & $1.16\! \times\!10^{-12}$\\  RGl-GMRES & 20 & 8.43 & 201 & $4.70\! \times\!10^{-04}$ & 262.64 & 2537 & $2.08\! \times\!10^{-11}$ & 331.46 & 2166 & $1.75\! \times\!10^{-11}$ \\ 
& 50 & 140.80 & 2516 & $9.37\! \times\!10^{-12}$ & 266.05& 2552 & $4.11\! \times\!10^{-12}$ & 437.85 & 2631 & $5.59\! \times\!10^{-12}$ \\  & 100 & 154.02 & 2677 & $3.97\! \times\!10^{-12}$ & 284.34& 2660 & $2.54\! \times\!10^{-12}$ &438.58 & 2747 & $2.61\! \times\!10^{-12}$ 
\\ 
\hline 
\end{tabular}
\end{table}
Table~\ref{Tabconvd2} demonstrates that RGl-GMRES method significantly reduces the computational cost compared with Gl-GMRES, particularly as the number of right-hand sides increases. For $\ell=50$ and $\ell=100$, RGl-GMRES method achieves residual norms of approximately $10^{-12}$, comparable to those of Gl-GMRES, while requiring substantially less CPU time. Although increasing the sketch size slightly increases the computational cost, it improves the accuracy of the computed solution. Overall, the results indicate that $\ell=50$ provides a favorable balance between computational efficiency and numerical accuracy.

\begin{exe}(Complex nonsymmetric convection--diffusion problem \cite{ZhangGu2013})
 We consider the multiple right-hand-side linear system (\ref{MRHS}), 
where $A\in\mathbb{C}^{n\times n}$ is a sparse, complex, nonsymmetric matrix and $B\in\mathbb{C}^{n\times s}$ contains $s$ right-hand sides. We set $N=100$ and $n=N^2=10,000$. Let $I=I_N$ denote the $N\times N$ identity matrix. The one-dimensional diffusion and convection matrices are defined, respectively, by
$$
D_2 =
\begin{pmatrix}
2      & -1     & 0      & \cdots & 0      \\
-1     & 2      & -1     & \ddots & \vdots \\
0      & -1     & 2      & \ddots & 0      \\
\vdots & \ddots & \ddots & \ddots & -1     \\
0      & \cdots & 0      & -1     & 2
\end{pmatrix}
\in\mathbb{R}^{N\times N}.
\quad D_1 =
\frac{1}{2}
\begin{pmatrix}
0      & 1      & 0      & \cdots & 0      \\
-1     & 0      & 1      & \ddots & \vdots \\
0      & -1     & 0      & \ddots & 0      \\
\vdots & \ddots & \ddots & \ddots & 1      \\
0      & \cdots & 0      & -1     & 0
\end{pmatrix}
\in\mathbb{R}^{N\times N}.
$$

The real convection--diffusion component is given by
$$
A_r =
D_2\otimes I
+
I\otimes D_2
+
\varepsilon_c
\left(
D_1\otimes I
+
I\otimes D_1
\right),
$$
where the convection parameter is set to $\varepsilon_c=8$. We then introduce a purely imaginary shift to obtain the complex coefficient matrix
$$
A=A_r+i\omega I_n,
$$
with $\omega=3$. Hence,
$$
A =
D_2\otimes I
+
I\otimes D_2
+
\varepsilon_c
\left(
D_1\otimes I
+
I\otimes D_1
\right)
+
i\omega I_n.
$$
 The matrix $B$ is generated as
$$
B=B_r+iB_i,
$$
where $B_r,B_i\in\mathbb{R}^{n\times s}$ are random matrices with independent standard normally distributed entries, where $s=\mathrm{NRHS}$. 
\end{exe}
\begin{table}[!htbp]
\centering
\caption{Comparison of Gl-GMRES and RGl-GMRES methods for different numbers of right-hand sides (Example 2).}
\vspace{0.5em}
\label{tab:results_rhs}
\small
\setlength{\tabcolsep}{3pt}
\renewcommand{\arraystretch}{1.15}
\begin{tabular}{l c|ccc|ccc|ccc}
\hline
\multirow{2}{*}{Method} & \multirow{2}{*}{$\ell$} &
\multicolumn{3}{c|}{ NRHS=50} &
\multicolumn{3}{c|}{ NRHS=100} &
\multicolumn{3}{c}{NRHS=150} \\
& & CPU & Iter. & Res
  & CPU & Iter. & Res
  & CPU & Iter. & Res \\
\hline
Gl-GMRES & -- 
& 80.29 & 898 & $1.13\!\times\!10^{-10}$
& 158.22 & 898 & $1.13\!\times\!10^{-10}$
& 251.40 & 898 & $1.13\!\times\!10^{-10}$ \\

RGl-GMRES & 20
& 33.51 & 853 & $1.96\!\times\!10^{-9}$
& 59.61 & 868 & $2.04\!\times\!10^{-9}$
& 83.70 & 855 & $1.88\!\times\!10^{-9}$ \\

        & 50
& 33.82 & 853 & $1.09\!\times\!10^{-9}$
& 64.56 & 869 & $1.03\!\times\!10^{-9}$
& 99.26 & 866 & $1.08\!\times\!10^{-9}$ \\

        & 100
& 38.66 & 888 & $5.38\!\times\!10^{-10}$
& 72.16 & 891 & $5.72\!\times\!10^{-10}$
& 126.07 & 883 & $5.05\!\times\!10^{-10}$ \\
\hline
\end{tabular}
\end{table}
Table~\ref{tab:results_rhs} compares the performance of Gl-GMRES and RGl-GMRES as the number of right-hand sides increases. The results demonstrate that RGl-GMRES consistently outperforms Gl-GMRES in terms of computational time for all tested values of $s$. Moreover, the computational savings become increasingly significant as the number of right-hand sides grows. For instance, when $s=150$, the CPU time is reduced from $251.40$ seconds for Gl-GMRES to only $83.70$ seconds for RGl-GMRES with $\ell=20$, corresponding to a reduction of approximately $66.7\%$. 

Despite this substantial reduction in computational cost, the number of iterations required by RGl-GMRES remains comparable to that of Gl-GMRES. Furthermore, the final relative residuals produced by RGl-GMRES remain below $2.1\times10^{-9}$ for all tested values of $\ell$, indicating that the randomized approach maintains a high level of numerical accuracy. Increasing the sketch dimension from $\ell=20$ to $\ell=100$ generally yields smaller residual norms; however, this improvement comes at the expense of increased CPU time. Overall, these results demonstrate that a relatively small sketch dimension, such as $\ell=20$, provides an effective trade-off between computational efficiency and numerical accuracy, particularly for linear systems with multiple right-hand sides \eqref{MRHS}.
%%%%%%%%%%%%%%%%%%%%%%%%
\subsection{Numerical Experiments for the Lyapunov Equation}
%%%%%%%%%%%%%%%%%%%%%%%%
Here we consider the large-scale continuous-time Lyapunov equation~\eqref{eq:lyapunov} with coefficient matrix $A$ obtained from a 5-point finite-difference discretization of the differential operator
\begin{equation*}
\mathcal{L}(u)
=
\Delta u
-f_1(x,y)\frac{\partial u}{\partial x}
-f_2(x,y)\frac{\partial u}{\partial y}
-g(x,y)u,
%\label{eq:operator2}
\end{equation*}
on the unit square $[0,1]\times[0,1]$, subject to homogeneous Dirichlet boundary conditions.

The coefficient functions are chosen as
\begin{equation*}
f_1(x,y)=\sin(x+2y),\quad
f_2(x,y)=e^y,\quad
g(x,y)=xy.
%\label{eq:coefficients}
\end{equation*}
The resulting sparse matrix $A$ has dimension
$
n=n_0^2
%\label{eq:dimension}
$,
where $n_0$ denotes the number of interior grid points in each spatial direction. In the numerical experiment, we set
\begin{equation}
n=10000,\qquad s=50.
\label{eq:parameters}
\end{equation}
The matrix $B\in\mathbb{R}^{n\times s}$ is generated with entries independently drawn from the uniform distribution on $[0,1]$. The initial approximation is chosen as $X_0=0$.

We compare the performance of the Lyapunov Gl-GMRES \cite[Algorithm 5.1]{JBILOU1999} and Lyapunov RGl-GMRES (i.e., Algorithm \ref{algo3}) methods. For both methods, the same convergence criterion is employed in order to ensure a fair comparison. Given an approximate solution $X_k$, the corresponding Lyapunov residual is defined by
\begin{equation}
R_k
=
-BB^T-AX_k-X_kA^T.
\label{eq:residual}
\end{equation}
The iterations are terminated when the relative Frobenius norm of the residual satisfies
\begin{equation}
\frac{\|R_k\|_F}{\|R_0\|_F}<10^{-7}.
\label{eq:stopping-global}
\end{equation}
This criterion is used for both algorithms so that their convergence and computational performance can be compared under the same prescribed accuracy.

For each method, we record the number of iterations required to satisfy
the stopping criterion, together with the corresponding computational
time and other relevant performance measures. The numerical results are
reported in Table \ref{tab:lyapunov_sketching}, where
\begin{equation*}
Res =\|R_k\|_{F}.
\end{equation*}
\begin{table}[htbp]
\centering
\caption{Comparison of the Lyapunov Gl-GMRES and Lyapunov RGl-GMRES methods for different sketching sizes.}
\vspace{0.5em}
%\label{tab:lyapunov_sketching}
%\small
\setlength{\tabcolsep}{6pt}
\renewcommand{\arraystretch}{1.15}
\begin{tabular}{l c c c c}
\hline
Method & $\ell$ & CPU & Iter. & Res \\
\hline
Lyapunov Gl-GMRES
& --
& 70.25
& 434  &$7.75 \!\times\!10^{-7}$\\
Lyapunov RGl-GMRES
& 20
& 1.47
& 420 &$9.98 \!\times\!10^{-7}$\\
& 60
&  1.51
&426&$9.68 \!\times\!10^{-7}$\\
& 100
& 1.62
& 430&$8.62 \!\times\!10^{-7}$ \\
\hline
\end{tabular}
\label{tab:lyapunov_sketching}
\end{table}
Table~\ref{tab:lyapunov_sketching} demonstrates the computational benefits of incorporating randomized sketching into the Gl-GMRES framework for solving the Lyapunov matrix equation. The standard Gl-GMRES method requires 70.25 seconds and 434 iterations to achieve a residual of $7.75\times10^{-7}$. By comparison, the RGl-GMRES method requires only 1.47, 1.51, and 1.62 seconds for sketching dimensions $\ell=20$, $60$, and $100$, respectively, with corresponding residuals of $9.98\times10^{-7}$, $9.68\times10^{-7}$, and $8.62\times10^{-7}$. Thus, the randomized formulation achieves a reduction in CPU time of more than an order of magnitude while maintaining residuals of the same order as those obtained by the standard method. The results also indicate that increasing the sketching dimension has only a moderate effect on the computational cost, while the convergence behavior remains largely unchanged. In particular, even the smallest sketch size considered, $\ell=20$, provides satisfactory accuracy at a substantially lower computational cost. This confirms the effectiveness of randomized sketching as a computationally efficient mechanism for reducing the cost of the Gl-GMRES process in large-scale Lyapunov problems.
\section{Conclusion}
\label{sec5}
We have developed a randomized variant of the Gl-GMRES method,
referred to as RGl-GMRES, for the efficient solution of large-scale
linear systems with multiple right-hand sides. The method incorporates
Gaussian sketching into the global Krylov subspace framework, allowing the
computational work associated with the matrix inner products to be
reduced while retaining the essential convergence properties of the
underlying Gl-GMRES procedure.

The numerical results presented in Section~\ref{sec4} show that RGl-GMRES can
provide substantial computational savings compared with the standard
Gl-GMRES method. In particular, the experiments indicate significant
reductions in CPU time, with savings approaching $50\%$ for several of
the tested configurations. The advantage of the randomized approach
becomes more noticeable as the number of right-hand sides increases,
which highlights its potential for large linear systems with multiple right-hand sides~\eqref{MRHS}. At the same time, the computed
solutions remain consistent with the prescribed convergence criteria,
indicating that the reduction in computational cost does not come at
the expense of an unacceptable loss of accuracy.

These observations suggest that randomized sketching can be an
effective tool for reducing the cost of global Krylov methods when
several right-hand sides have to be treated simultaneously. Further
investigation is nevertheless needed to assess the behavior of
RGl-GMRES for a broader range of problem classes, matrix dimensions,
sketch sizes, and stopping tolerances.

An interesting extension of this work is the application of the
randomized global Krylov subspace framework to matrix equations, including
Sylvester and Lyapunov matrix equations as well as more general linear matrix
equations. Such problems possess additional block and, in many
applications, low-rank structures that may be exploited within a
randomized projection framework. Developing appropriate randomized
global Krylov subspace algorithms for these settings could provide further reductions
in computational and storage costs while preserving satisfactory
convergence and accuracy.
%%%%%%%%%%%%%%%
\section*{Acknowledgment}
The work of X.-M. Gu was supported by the Key Laboratory of Large-Scale Electromagnetic Industrial Software, Ministry of Education (EMCAE202502).
\section*{Ethics declarations}
\subsection*{Conflict of interest}
The authors declare no competing interests.
\subsection*{Data availability}
The data that support the findings of this study are available in the article. The source code used in this study is available from the authors upon reasonable request.
\subsection*{Contribution statement}
A. Badahmane: Conceptualization (original idea), Methodology, Formal analysis, Software, Validation,
Investigation, Data curation, Visualization, Writing – original draft, Writing – review \& editing.
I. Ezzagouri: Formal analysis,  Writing – original draft.
Xian-Ming. Gu: Formal analysis,  Writing – original draft, Writing – review \& editing, 
%%%%%%%%%%%%%%%%%%%%%%%%%%%%
\bibliographystyle{elsarticle-num-names}
%\begin{thebibliography}{1}
%\bibliographystyle{elsarticle-num}   % 设定引用样式，如 plain, unsrt, ieeetr 等。很多期刊会提供自己的 .bst 文件，比如 \bibliographystyle{etna}
\bibliography{references}   % 这里写你的 .bib 文件名，**不需要**加 .bib 后缀

\end{document}